\documentclass[a4paper,12pt,twoside]{article}
\usepackage[latin1]{inputenc}
\usepackage[dvips]{graphics}
\input xy
\input cyracc.def

\xyoption{all}
\usepackage{xr}
\usepackage{amssymb}
\usepackage{amsmath}
\usepackage{amsthm}
\usepackage{amsfonts}
\usepackage{mathrsfs}
\theoremstyle{plain}
\newtheorem{T}{Theorem}[section]

\newtheorem{Cor}[T]{Corollary}

\newtheorem{Prop}[T]{Proposition}
\newtheorem{PropDef}[T]{Proposition/Definition}

\newcommand\mps[1]{\marginpar{\small\sf }}

\theoremstyle{definition}\newtheorem{D}{Definition}[T]
\theoremstyle{remark}
\theoremstyle{remark}

\newcommand{\bb}[1]{\mathbb{#1}}
\newcommand\calo{\mathcal O}

\newcommand{\mc}[1]{\mathcal{#1}}

\newcommand{\Hom}{\mathrm{Hom}}

\newcommand\op[1]{\operatorname{#1}}

\newcommand\rk{\operatorname{rk}}

\input cyracc.def

\theoremstyle{plain}

\theoremstyle{definition}

\newcommand{\Spec}{\mathrm{Spec\,\,}}

\newcommand{\ses}[5]{$$
\xymatrix@1{ 0 \ar[r] & {{#1}} \ar[r]^-{{#2}} & {{#3}} \ar[r]^-{{#4}} &
{{#5}} \ar[r] & 0
\\ }$$}

\newcommand{\sesdot}[5]{$$
\xymatrix@1{ 0 \ar[r] & {{#1}} \ar[r]^-{{#2}} & {{#3}} \ar[r]^-{{#4}} &
{{#5}} \ar[r] & 0. \\ }$$}

\newcommand{\sesbig}[5]{$$\xymatrix@1{{\raisebox{1.0ex}[3.0ex][1.0ex]{$0$}}
\ar@<0.6ex>[r] & {\raisebox{1.0ex}[3.0ex][1.0ex]{${#1}$}}
\ar@<0.6ex>[r]^{#2} & {\raisebox{1.0ex}[3.0ex][1.0ex]{${#3}$}}
\ar@<0.6ex>[r]^{#4} & {\raisebox{1.0ex}[3.0ex][1.0ex]{${#5}$}}
\ar@<0.6ex>[r] & {\raisebox{1.0ex}[3.0ex][1.0ex]{$0$}}}$$}

\newcommand{\Comment}[1]{}

\title{Degenerating Riemann surfaces and the Quillen metric}
\author{Dennis Eriksson}
\begin{document}
\maketitle
\abstract{The degeneration of the Quillen metric for a one-parameter family of Riemann surfaces has been studied by Bismut-Bost and Yoshikawa. In this article we propose a more geometric point of view using Deligne's Riemann-Roch theorem. We obtain an interpretation of the singular part of the metric as a discriminant and the continuous part as a degeneration of the metric on Deligne products, which gives an asymptotic development involving the monodromy eigenvalues. This generalizes the results of Bismut-Bost and is a version of Yoshikawa's results on the degeneration of the Quillen metric for general degenerations with isolated singularities in the central fiber.}


\tableofcontents
\newpage

\section{Introduction} The Quillen metric is a natural metric on the determinant of the cohomology of a hermitian vector bundle on a compact K\"ahler manifold $X$. When varying the metric in smooth proper families $X \to S$, this moreover defines a smooth metric on $S$. However, whenever one approaches a singular fiber, the metric will generally be singular. The first article to treat refined invariants of this singular metric seems to be \cite{Bismut-Bost} which considered degenerations of Riemann-surfaces into singular fibers which has at worst nodal singularities. The current article arose as an attempt to understand a comment in \emph{idem}, section 2(e), concerning how it might be possible to use a theorem of Deligne, in a special case, to deduce the same result. It seems that using this theorem of Deligne, combined with theorems of T. Saito \cite{T.Saito-conductor} and D. Barlet \cite{Barlet}, allows us to give analogous results to those of Bismut and Bost, but for arbitrary singularities in the special fiber. \\ More precisely, suppose $S$ is the unit disc $D$, and suppose we have a family of Riemann surfaces $X \to D$ with singular central fiber (satisfying some additional conditions, see Theorem \ref{sing-of-metric}). If $\sigma$ is a local trivialization of the determinant of the cohomology of a (hermitian) vector bundle $E$, there is the following description of the singularity of the Quillen metric around $0$, $t \in S$:
$$\log|\sigma|_Q(t) = \rk E \frac{\Delta_f}{12}\log|t| + \varphi$$
where $|\sigma|_Q$ denotes the Quillen norm of the section $\sigma$, $\Delta_f$ is a discriminant number and $\varphi$ is a continuous function. If moreover $X$ is smooth and $f^{-1}(0) = X_0$ has only isolated singularities we obtain an asymptotic expansion (see Theorem \ref{thm:asymptotic}):
 $$\log|\sigma|_Q(t) \sim \rk E \frac{\mu_f}{12}\log|t| + \varphi_0 + \sum_r T_{m,m'}^{r,h} t^m \overline{t}^{m'} |t|^r (\log |t|)^h,$$ where $\mu_f$ is the total Milnor number of $X_0$, $\varphi$ a smooth function, $r$ are local monodromy eigenvalues of the singularities, and $T_{m,m'}^{r,h}$ are certain constants depending on the metric on $E$ further indexed by $m,m', h \in \bb N$. \\
These results seem to be related to results of Yoshikawa in \cite{Yoshikawa-Smoothinghypersurfacesing}, where he considers the metric induced from the tangent bundle on $X$ when $X$ is smooth, to give formulas for the curvature of the Quillen metric, also in higher relative dimension than one. These theorems are surely related, but the necessary comparison has not been made. Unlike Yoshikawa's result, however, \ref{sing-of-metric} does not seem to obtain contributions from non-reduced components in the special fiber. \\
The article is organized as follows. After reviewing some known results we state the main theorem. In the following section we study the degeneration of the Deligne-isomorphism, and relate it to discriminants in the sense of projective duality (cf. \cite{GZK}, chapter 1). We then apply the above mentioned result of T. Saito (and a slight refinement from \cite{discriminant-conductor}) to control the singular part of the metric. Finally we apply the result of Barlet to control the continuous part of the metric which is controlled by Deligne products. The general approach is geometric, and shows which terms correspond to which part of the metric. It seems however difficult to obtain similar geometric results in higher relative dimension. \\

$\mathbf{Acknowledgements}$: I wish to thank Gerard Freixas  for many interesting comments on various topics of this article and Bo Berndtsson for valuable discussions. I'm also very grateful for the many explainations Jan Stevens shared with me on monodromy and curve singularities. Finally I'm happy to acknowledge the encouragement of Vincent Maillot to consider the case of degenerations when the total space is not smooth.
\section{Definitions and Recollections}
In this section we will review the definition and properties of the relevant objects for this article. Many of the definitions have much more general contexts, but for the purpose of this article we restrict ourselves to the situation considered here and to simplify the exposition we suppose unless stated that $X$ is smooth. Good references for this theory is \cite{determinant} and \cite{Soule-bourbaki}. All vector bundles considered will also be holomorphic vector bundles.\\
Suppose $X$ is a connected smooth complex surface, and let $f: X \to D$, with the unit disc $D = \{z \in \bb C, |z| < 1\}$, be a proper holomorphic map which is a submersion outside of $X_0 = f^{-1}(0)$. To relax language, we say that $X \to D$ is a relative curve. \\
Define the relative canonical bundle as $\omega = \det \Omega_X \otimes \det T_D$, where $\Omega_X$ (resp. $T_D$) denotes the holomorphic cotangent bundle of $X$ (resp. the holomorphic tangent bundle on $D$). It can more generally be defined when $f$ is a local complete intersection. Given a vector bundle $E$ on $X$, we define a line bundle on $D$ whose fibers over every point $t \in D$ is given by $\det H^0(X_t, E_t) \otimes \det H^1(X_t, E_t)^{-1}$ for sheaf cohomology groups $H^i(X_t, E_t)$. It is denoted by $\lambda(E)$ and is called the determinant of the cohomology. A more global definition is provided by applying the Knudsen-Mumford-determinant (cf. \cite{Knudsen-I}) to the direct image sheaves $f_* E$ and $R^1f_* E$ which exhibits this line bundle as a holomorphic line bundle on the base. Whenever $E$ and $\omega$ have $C^\infty$ hermitian metrics (the notation $(E,h)$ or $\overline {E}$ will be used to denote a hermitian metric on a vector bundle $E$), the Quillen metric is defined on $\lambda(E)_t$ where $t \neq 0$. It is defined as follows: First of all, for every fiber $X_t$ the groups $H^i(X_t, E_t)$ can, by Hodge theory, be represented by harmonic forms and the induced metric on $\lambda(E)_t$ is the $L^2$-metric. Let $$\zeta_\Delta(s) = \sum_{\lambda > 0} \frac{1}{\lambda^s}$$
where the sum is over all positive eigenvalues $\lambda$ of the Kodaira-Laplace operator $\Delta_{\overline \partial}$ acting on $E_t$-valued smooth functions. Then the Quillen metric is:
$$h_Q = h_{L^2} \exp(\zeta'_\Delta(0)).$$
The Quillen metric is a smooth metric on $\lambda(E)$ outside of $0$, and the purpose of this article is to study the behavior of this metric close to 0. \\
We recall also the main result of Deligne we will need, which also explains the general strategy of this article. Given two line bundles $L$ and $M$ on $X$ for a relative curve $X$ over a complex manifold $S$, we recall that there is a natural line bundle on $S$, the Deligne-product. It can be defined as $$\langle L, M \rangle := \lambda((L-1)\otimes (M-1))$$
where $\lambda = \det Rf_* $ and $\lambda(A \pm B) := \lambda(A) \otimes \lambda(B)^{\pm 1}$. Moreover, given two smooth hermitian metrics on $L$ and $M$, the line bundle $\langle L, M \rangle$ carries a natural metric \cite{determinant}, \cite{Elkik-metric}. In this article we will only need two properties of this metric (for which the previous two articles are good references): \\
\begin{enumerate}
 \item  If $X \to D$ is a holomorphic submersion, there is the following curvature formula: $$c_1(\langle L , M \rangle) = \int_{X/D} c_1(\overline L) \wedge c_1(\overline M).$$
 \item  In the same situation, if $\overline L \simeq \overline L'$ is an isomorphism of line bundles, then the norm of the induced isomorphism $$\langle \overline L, \overline M \rangle \simeq \langle \overline L', \overline M \rangle $$ is given by $a(t) = \exp \int_{X_t} \log (h/h') c_1(\overline M)$, for $h$ and $h'$ denoting the respective metrics on $\overline L$ and $\overline L'$.
\end{enumerate}
We will also introduce another line bundle, $IC_2(E)$ on $D$. In a similar spirit, it is defined as $$IC_2(E): = \lambda(E - \rk E - (\det E - 1))$$
which should be interpreted in the same way as in the previous definition. Denote by $\overline E$ a hermitian metric on a vector bundle $E$. The following properties characterize the line bundle $IC_2(\overline E)$ and the metric on it:
\begin{enumerate}
  \item If $L$ is a line bundle on $X$, then $IC_2(L)$ is the trivial line bundle with trivial metric.
  \item If $X \to D$ is a holomorphic submersion, there is the following curvature formula: $$c_1(IC_2(\overline E)) = \int_{X/D} c_2(\overline E).$$
  \item \label{whitney}  [Whitney formula] Given a short exact sequence of vector bundles $$\mc E: 0 \to E' \to E \to E'' \to 0,$$ there is an isomorphism
  $$IC_2(E) \simeq IC_2(E') IC_2(E'') \langle \det E', \det E'' \rangle$$
  and if the various vector bundles carry hermitian metrics, this has norm $\exp(a)$ with $$a(t) = \frac{1}{2\pi i}\int_{X_t} \op{Tr}(\alpha^* \wedge \alpha)$$ and $\alpha \in \Omega^{0,1} \Hom(E'', E')$ is the smooth vector valued $(0,1)$-form determining the extension $\mc E$ with metrics.
\end{enumerate}
Whenever $X \to S$ is not a submersion but at least projective, the metrics in question are continuous by the arguments of \cite{Moriwaki-Deligne}, and the above formulas for the curvature have to be interpreted in the sense of currents. We will see in section \ref{monodromy} how to refine this result. The importance of these line bundles come from the following version of the Riemann-Roch theorem by Deligne.
\begin{T} [\cite{determinant}, Th\'eor\`eme 11.4, \cite{Soule-bourbaki}, 4.4]\label{thm:riemann-roch} Let $X \to S$ be a holomorphic submersion whose fibers are Riemann surfaces, and let $\overline{E}$ (resp. $\overline{\omega}$) be a hermitian vector bundle (resp. a hermitian metric on $\omega$) on $X$. Then there is a canonical isomorphism, up to sign,
$$\lambda(E)^{12} \simeq \langle \omega, \omega \rangle^{\rk E} \langle \det E, \det E \otimes \omega^{-1}\rangle^6 IC_2(E)^{-12}.$$
Moreover, when the left side is equipped with the Quillen metric, and the right side is equipped with the metrics above, the logarithm of the norm of this isomorphism is
$$\rk E \cdot (2-2g) \cdot \left(\frac{\zeta'(-1)}{\zeta(-1)} + \frac{1}{2}\right).$$
Here $\zeta(s)$ denotes the Riemann zeta-function.
\end{T}
Partially following \cite{T.Saito-conductor} and \cite{Deligne-Quillen}, the induced rational section $$\Hom(\langle \omega, \omega \rangle^{\rk E} \langle \det E, \det E \otimes \omega^{-1}\rangle^6 IC_2(E)^{-12}, \lambda(E)^{12})$$ in generically smooth families of Riemann surfaces is called the discriminant section (see Proposition \ref{Prop:discriminant} for an explanation of this).

\begin{D}\label{def:smoothmetric}Whenever $X$ is not smooth, we consider a desingularization $\pi: X' \to X$, i.e. a proper holomorphic map which is an isomorphism over the smooth locus. We define a $C^\infty$ metric on a vector bundle $E$ on $X$ a hermitian $C^\infty$ metric on $\pi^* E$ for some desingularization $\pi: X' \to X$. Two metrics associated to desingularizations $X' \to X$ and $X'' \to X$ are said to be equivalent if there is a third desingularization $X'''\to X$ which dominates the two other desingularizations such that the metrics coincide when pulled back to $X'''$. \end{D}
\begin{Prop} \label{prop:deligne-metric} Let $\pi: X' \to X$ be a resolution of singularities and $\overline{L}$ and $\overline{M}$ line bundles with $C^\infty$ metrics. Then
the $C^\infty$ metric on $\langle \overline{L}, \overline{M} \rangle$ on $D^*$ has a continuous extension to 0. Likewise for a vector bundle $\overline{E}$ with a $C^\infty$ hermitian metric, the metric on $IC_2(\overline{E})$ has a continuous extension to 0. This is independent of the equivalence class. 
\end{Prop}
\begin{proof} The proof of \cite{discriminant-conductor}, Theorem 1.4 shows that $$\langle L, M \rangle \simeq \langle \pi^* L, \pi^* M \rangle.$$ The arguments of the main result of \cite{Moriwaki-Deligne} then shows that the metric on the latter bundle is continuous, since the family is projective by Proposition \ref{Prop:algebraic}. Since the metric is already smooth outside of 0, the continuous extension is necessarily unique and independet of the desingularization. For the second point, the argument of \cite{Bismut-Bost} allows us to filter $E$ by vector bundles with line bundle quotients. Then the Whitney formula for $IC_2$ (see property (\ref{whitney}) of $IC_2$) shows that
$IC_2(E) \simeq IC_2(\pi^* E)$. Using the result for line bundles, and \cite{Barlet-continuity} for continuity of the fiber integrals involved, we see that these metrics are also continuous.
\end{proof}

We also recall the following statement:
\begin{PropDef} [\cite{Steenbrink} 2.26, also see \cite{laufer}, \cite{wahl}, 3.15.1] Suppose $X$ normal local complete intersection surface, admitting a smoothening $p: Y \to D$ with $p^{-1}(0) = X$. Then the Milnor number of a point $x \in X$ is independent of the smoothening and is given by
$$\mu_{X,x} = 12p_g + \Gamma^2 - b_1(E) + b_2(E).$$
These numbers are defined as follows. Let $\pi: X' \to X$ be a resolution of singularities of $x$, with exceptional set $E$. Then $p_g = \dim_\bb C R^1 \pi_* \calo_{X'}$, $\Gamma$ is the discrepancy
$$\Gamma = \omega_{X'/D} - \pi^* \omega_{X/S},$$ and $b_i(E)$ are the Betti numbers of $E$.
\end{PropDef}
In the general case \cite{Steenbrink} remarks that the above number is independent of choice of resolution of singularities, and we use the same formula to define, for the purposes of this article, the Milnor number $\mu_{X,x}$ and put $\mu_X = \sum_{x \in X} \mu_{X,x}$. \\ We are now ready to state the main result. We consider as before $f: X \to D$ a proper holomorphic map which is a submersion outside of $t \neq 0$ with Riemann surface fibers, $X$ connected. We suppose that we are given a vector bundle $E$ on $X$, and that $E$ and $\omega$ (see above for the definition) are equipped with $C^\infty$ hermitian metrics. 
\begin{T} \label{sing-of-metric} Suppose that $X$ is normal, and that $X \to D$ is a projective local complete intersection. Let $\sigma$ be a local section of $\lambda(E)$ on $D$ such that $\sigma(0) \neq 0$. Then
$$\log|\sigma|_Q = \rk E \frac{\Delta_f}{12}\log|t| + \varphi$$
where $\varphi$ is continuous, and $$\Delta_f = \mu_X + \chi(X_0) - \chi(X_t),$$
with $\chi$ denoting topological Euler characteristic and $t \in D \setminus \{0\}.$
\end{T}
In the case when $X$ is smooth we also obtain the following version of \cite{Yoshikawa-Smoothinghypersurfacesing}:
\begin{T} \label{thm:asymptotic} Suppose furthermore that $X$ is smooth (so that $f$ is automatically a projective local complete intersection, see Proposition \ref{Prop:algebraic} below), and that $X_0$ has only isolated singularities. Then we have an asymptotic expansion at 0:
$$\log|\sigma|_Q \sim \rk E \frac{\mu_f}{12}\log|t| + \varphi_0 + \sum T_{m,m'}^{r,h} t^m {\overline t}^{m'}|t|^r (\log|t|)^h.$$
Here $\mu_f = \chi(X_0) - \chi(X_t)$ denotes the Milnor number of $X_0$, $\varphi_0$ is a smooth function, $T_{m,m'}^{r,h}$ are constants, $m,m', h$ are non-negative numbers, and $r \in (0,2)\cap \bb Q$ are eigenvalues of the monodromy acting on the local Milnor fibers.
\end{T}
In the following two sections we will consider on one hand the singular term which contains the discriminant number $\Delta_f$, and in the next section the continuous part which admits the above asymptotic expansion.
\section{Singularity of the Quillen metric and discriminants}
The main result of this section relates the singularity of the Quillen metric to the Deligne discriminant (Proposition \ref{Prop-saito:discriminant}). We also connect it to discriminants in the sense of projective duality (Proposition \ref{Prop:discriminant}). \\ Suppose $X$ is a normal complex surface (e.g. complex dimension 2 with isolated singularities) and let $f: X \to D$ with the unit disc $D = \{z \in \bb C, |z| < 1\}$ be a proper holomorphic map which is a submersion outside of $D^* = D \setminus 0$, such that $f$ is furthermore projective and a local complete intersection. Consider the ring of holomorphic germs $R = \bb C\{t\}$ around 0 consisting of powerseries with positive radius of convergence and its spectrum $S = \Spec R$. We will need a version of \cite{Bismut-Bost}, Proposition 3.4, to make the situation algebraic enough that we can apply results which are usually stated for schemes:
\begin{Prop} \label{Prop:algebraic} If $X$ is smooth the condition that $f$ is projective and a local complete intersection is automatic, possibly after shrinking $D$. If the family $X \to D$ is already projective, it defines a projective scheme over $\Spec R$, denoted by $\mc X$. Any coherent sheaf on $X$ gives rise to coherent sheaf on $\mc X$.
\end{Prop}
\begin{proof} First suppose $X$ is smooth and let $C_i, i = 1, \ldots, k$ be the reduced irreducible components of $X_0$, and let $p_i$ be a smooth point on each $C_i$. Then for a small open neighborhood $U_i$ of 0 we obtain a section $P_i: U_i \to X$. Replacing $D$ with a smaller disc we thus have sections $P_i: D \to X$. By \cite{Grau}, Satz 4, the line bundle $\calo(\sum p_i)$ is ample on $X_0$. One concludes as in \cite{Bismut-Bost}, Proposition 3.4 that for some $n$, $L = \calo(\sum nP_i)$ is very ample with respect to $f$. Now assume $X \to D$ is already projective. Because $D$ is Stein, the cohomology sheaves $f_* L^i$ are finite type modules over the ring of holomorphic functions on the disc, and $\oplus f_* L^i$ then forms a homogenous graded ring over the ring of holomorphic functions on the disc. Tensoring this ring with $R$ gives a homogenous graded ring over $R$ whose Proj is the sought after scheme. The statement about coherent sheaves follows the same lines. Finally, the condition that $f$ is a local complete intersection is automatic if the varieties are smooth.
\end{proof}
In the above situation, Deligne's isomorphism induces an abstract isomorphism over $D$:
$$\lambda(E)^{12} \simeq \langle\omega, \omega \rangle^{\op{rk} E}  \langle \det E, \det E \otimes \omega^{-1} \rangle^6 IC_2(E)^{-12} \calo([0] \cdot n)$$
where $n$ is a natural number and the sheaf $\calo([0] \cdot n)$ denotes the sheaf of holomorphic functions with a zero of order at least $n$ at 0. This sheaf will be responsible for the singularity of the Quillen metric at 0 and we intend to give a geometric interpretation of this singularity in terms of projective duality. For this, denote by $\check {\bb P}^N$ the dual projective variety whose points parametrize hyperplane sections in $\bb P^N$. The discriminant variety, $\Delta_X \subseteq \check {\bb P}^N$ of $X \subseteq \bb P^N$, where we suppose $X$ is not contained in any linear subspace, is the variety of hyperplane sections $H$ such that $H \cap X$ is singular. Whenever $X$ is a smooth complex algebraic surface (i.e. complex dimension is 2) of degree at least 2 in $\bb P^N$ a theorem of Ein (cf. \cite{Ein}) says that $\Delta_X \subseteq \bb P^N$ is a hypersurface, in particular it is given by a polynomial (the discriminant) defined up to a constant. We can apply the Deligne-Riemann-Roch isomorphism to the tautological family of hyperplane sections $\mc H \to \check {\bb P}^N$.
\begin{Prop} \label{Prop:discriminant} [compare \cite{discriminant-conductor}, Proposition 3.1] In the above situation, let $\omega = \omega_{\mc H/\check {\bb P}^N}$ be the relative dualizing sheaf. Then the Riemann-Roch isomorphism over $\check {\bb P}^N \setminus \Delta_X$ extends to an isomorphism
$$\lambda(\omega)^{12} \simeq \langle \omega, \omega \rangle \otimes \calo(\Delta_X)$$
over $\check {\bb P}^{N}$ so the discriminant section is represented by the discriminant hypersurface.
\end{Prop}
\begin{proof} Since the Riemann-Roch discriminant is an isomorphism outside of the discriminant variety and the discriminant variety is an integral variety we immediately have that the degeneration is measured by $\calo(k \Delta_X)$. We want to prove that $k = 1$ by using a Lefschetz pencil $\bb P^1 \subseteq \check {\bb P}^N$. The restriction of the isomorphism to any such pencil $\bb P^1$ gives a family of Riemann surfaces whose singular fibers have exactly one singular point which is a non-degenerate quadratic singularity. In this case it is well known on one hand that the order of zero of the discriminant is 1, and on the other that the order of degeneration of the Deligne isomorphism is 1 and we conclude that $k = 1$.
\end{proof}
Thus, in the case of families of hyperplane sections of a surface we have an interpretation of the Riemann-Roch discriminant as an "actual discriminant". This allows us to study, in this special case, the degeneration. Let $X$ be as above, and suppose that we are given a singular hyperplane section $H$ with isolated singularities. Pick any other hyperplane $H'$ such that $\# \{H \cap H'\} = \deg X$. The base locus of $X \to \bb P^1$ given by $x \mapsto [H(x): H'(x)]$ is $H \cap H'$ which is reduced and smooth. Then blowing up $H \cap H'$ produces a smooth surface $X'$ whose map to $\bb P^1$ gives a family of hyperplane sections. In this particular case it is well known (cf. \cite{Dimca-Milnor}) that the vanishing of the discriminant, and thus by the above proposition the degeneration of the Deligne-isomorphism, is given by the Milnor number of the singular fiber. Hopefully the above geometric considerations will shed some light on the following general statement:
\begin{Prop}[\cite{T.Saito-conductor}, \cite{discriminant-conductor}, Theorem 1.4] \label{Prop-saito:discriminant} Let $E$ be a vector bundle on $X$, and \linebreak $X \to D$ a projective local complete intersection. Then the order of degeneration of the Riemann-Roch-discriminant at 0 is $\rk E \cdot \Delta_f$, where $$\Delta_f = \mu_X + \chi(X_0) - \chi(X_t).$$ Here $\mu_X$ denotes the Milnor number of $X$,  $\chi$ denotes the topological Euler-characteristic and $X_t$ is any fiber over $t \neq 0$.
\end{Prop}
\begin{proof} The result in \emph{loc. cit. } states that the order of the degeneration of the Riemann-Roch-discriminant is $\Delta_f$, when $E = \omega$ or $\calo_X$. We have to show that the general result follows from this. As in \cite{determinant} one applies Serre-duality, which is valid for local complete intersection $f$, and see that Riemann-Roch isomorphism gives a global isomorphism over $D$
$$\det Rf_* (L - \omega)^2 \simeq \langle L \otimes \omega^{-1}, L \otimes \omega^{-2} \rangle$$
without any degeneration. In general, by the projectivity assumption, the argument of \cite{Bismut-Bost}, Proposition 3.5, allows us to filter any vector bundle by vector bundles whose successive quotients are line bundles. The result follows from an apparent additivity on both sides.
\end{proof}
Now, by the second part of Theorem \ref{thm:riemann-roch} the norm of the Deligne-isomorphism 
$$\det Rf_* (E)^{12} \simeq  \langle\omega, \omega \rangle^{\op{rk} E}  \langle \det E, \det E \omega^{-1} \rangle^6 IC_2(E)^{-12} \calo(\rk E \Delta_f \cdot [0])$$
is constant in the family outside of $t = 0$, so the study of the Quillen metric close to 0 is equivalent to that of the right hand side. We note that there is a natural singular metric on $\calo(\rk E \Delta_f \cdot [0])$. Indeed, consider the section $t^n$ of $\calo(n \cdot [0])$ and put $\|t^n\| = |t|^n$. Then this defines a singular metric on $\calo(n \cdot [0])$ whose curvature is $n \cdot \delta_0$. The metrics on the Deligne bundles is continuous at 0 by Proposition \ref{prop:deligne-metric}. It then follows that if $\sigma$ is a local holomorphic frame at $0$ of $\lambda(E)$, then $$\log|\sigma|_Q - \frac{\rk E \cdot \Delta_f}{12} \log|t|$$ is continuous at 0. This proves Theorem \ref{sing-of-metric}. \\
In the next section we will follow the method of \cite{Bismut-Bost}, but apply at a certain moment \cite{Barlet} to obtain more precise results on the continuity when the total space $X$ is smooth and only having isolated singularities in $X_0$. 

\section{The contribution from monodromy}\label{monodromy}
In this section we suppose that the total space $X$ is smooth, that $X \to D$ is a holomorphic submersion outside of finitely many points in $X_0$. Around each singular point, consider locally the map $(\bb C^2,0) \to (\bb C, 0)$. Given a small $\epsilon > 0$, consider the sphere $S^3_\epsilon$ of radius $\epsilon$ around 0, denote by $V(f) = \{f = 0\}$ the vanishing set of $f$ and set $K = V(f) \cap S$.  Then, for small $\epsilon$,  $\phi: S^3_\epsilon \setminus K \to S^1$ given by $z \mapsto f(z)/|f(z)|$, is by Milnor \cite{Milnor} a fibration, and for $p \in S^1$, $F_f = \phi^{-1}(p)$ has the homotopy-type of a bouquet of $\mu$ spheres, where $\mu$ is the Milnor number of the singularity. We also have $\mu = \dim_\bb C \bb C \{x,y \}/(\partial f/\partial x, \partial f/\partial y)$. The monodromy-action of $T$ is then the induced action of going around $S^1$ on $H_1(F_f) = \bb Z^\mu$. The eigenvalues of $T$ are known to be of the form $\exp(\pi i r)$ for rational numbers $r$. They appear as roots of the local Bernstein-Sato-polynomial, by results of Malgrange (see in particular \cite{Malgrange}). Barlet \cite{Barlet} then relates the roots of the Bernstein-Sato polynomial to asymptotic developments of integration along fibers of forms, of which the following is a special case:
\begin{T}[Barlet, \cite{Barlet}] Let $f: (\bb C^2, 0) \to (\bb C, 0)$ be the germ of a holomorphic function, and let $\phi$ be a smooth $(1,1)$-form with compact support on $\bb C^2$. Then
$F_\phi: D \to \bb C$ given by 
$$t \mapsto \int_{X_t} \phi$$ has the following asymptotic expansion:
$$F_\phi \sim  \sum_{r \in \{r_1, \ldots, r_k\}} \sum_{m,m' \in \bb N_+} \sum_{h \in \bb N_+} T_{m,m'}^{r,h}(\phi) t^m \overline{t}^{m'}|t|^{r}\log|t|^h$$
for certain $(1,1)$-currents $T_{m,m'}^{r,h}$ on $\bb C^2$. The numbers $r_i \in (0,2) \cap \bb Q$ appear as monodromy eigenvalues in the local Milnor fiber.  \\
\end{T}
We aim to study the behavior of the metrics on the Deligne bundles $\langle , \rangle, IC_2$ close to the singular fiber. We treat the bundle $\langle, \rangle$, $IC_2$ is treated in the same way. Suppose first that $\overline L, \overline M$ are two hermitian line bundles on $X$, where $X \to D$ is as in the theorem, such that the curvature forms are zero in a neighborhood of the singular points. This can always be found by a partition of unity. Then the curvature \linebreak $c_1(\langle \overline L, \overline M \rangle) = \int_{X/D} c_1(\overline L) \wedge c_1(\overline M)$ is also smooth, and thus the metric must necessarily be smooth. Now, pick any such metric on $L$, denoted by $h$, and $h'$ any other metric on the same line bundle, and let $\overline M$ be any hermitian line bundle. The norm of the identity-map $\langle (L,h), \overline M \rangle = \langle (L, h'), \overline M \rangle$ is equal to $$\int_{X/D} \log (h/h') c_1(\overline M),$$
which is a fibre integral as above. Thus we can compare any Deligne metric on $\langle L, M \rangle$ with a smooth one, the difference being a series of fiber integrals.  In other words:
\begin{Cor} Let $\sigma$ be a local section of the hermitian line bundle $\langle \overline L, \overline M \rangle$ such that $\sigma(0) \neq 0$. Then $\log|\sigma|$ has an asymptotic expansion at 0 of the form described above. Moreover, by the same argument the same statement holds for $IC_2(\overline E)$.
\end{Cor}
This concludes the proof of Theorem \ref{thm:asymptotic}. 
\section{The example of degenerating elliptic curves}
We finally wish to give an example which is not covered by the general theory of \cite{Bismut-Bost}. In this example we use various standard facts of Weierstrass models from the theory of elliptic curves. Let $f: X \to D$ in $\bb P^2 \times D$ be given by a Weierstrass equation $$y^2 + a_1xy + a_3y= x^3 + a_2x^2 + a_4 x + a_6$$ where $a_i$ are holomorphic functions in $D$. It has associated discriminant $\Delta$, which for $a_1 = a_2 = a_3 = 0$ is given by $4a_4^3 + 27 a_6^2$. We suppose it is not identically zero, in which case it can be shown that $\Delta_f$ is the order of vanishing at 0 of $\Delta$, and that the order of vanishing at 0 is less than 12 (a minimal Weierstrass equation). This family then satisfies the conditions of Theorem \ref{sing-of-metric}. The general fiber is an elliptic curve and $X$ has at worst a single singularity which is a du Val (or A-D-E-singularity), and the invariant differential $$\omega = \frac{dy}{3x^2 + 2a_2 x + a_4 - a_1y} = \frac{dx}{2y + a_1 x + a_3}$$ defines a global trivialization of $\omega_{X/D} = f^* f_* \omega_{X/D}$. The N\'eron model is the minimal resolution of singularities, $\pi: \mc E \to X$, and $\pi^* \omega$ also globally trivializes $\omega_{\mc E/D} = \pi^* \omega_{X/D}$. Thus any smooth function $\mc E \to \bb R$ defines a smooth metric on $\omega_{X/D}$ in the sense of Definition \ref{def:smoothmetric}. For the 0-section $[0:1:0] \times D$ of the Weierstrass model we have $f^* 0^* \omega_{\mc E/D} =  \omega_{\mc E/D}$ and the restriction of the invariant differential allows us to define a hermitian metric on $\omega_{\mc E/D}$ by choosing a smooth function $D \to \bb R$. Picking a constant function defines a fiberwise flat metric on $\omega_{\mc E/D}$, whose analytic torsion is given by a real-analytic Eisenstein series and content of Theorem \ref{sing-of-metric} is basically the famous Kronecker limit formula. \\
 As a concrete example, if $y^2 = x^3 + t^k$ for $1 \leq k \leq 5$, then $\Delta = 27 t^{2k}$, $\Delta_f = 2k$, and $X$ is smooth if and only if $k=1$. In this case the Milnor number is given by $2$ and the monodromy operator is given by $$T = \begin{bmatrix}
0 & 1 \\
-1 & 1
\end{bmatrix}$$ with monodromy eigenvalues given by primitive sixth roots of unity.

\providecommand{\bysame}{\leavevmode\hbox to3em{\hrulefill}\thinspace}
\providecommand{\MR}{\relax\ifhmode\unskip\space\fi MR }
\providecommand{\MRhref}[2]{%
  \href{http://www.ams.org/mathscinet-getitem?mr=#1}{#2}
}
\providecommand{\href}[2]{#2}

\end{document}